\documentclass[12pt,a4paper]{article}%
\usepackage{amsmath, amsfonts, amsthm,color, latexsym, amssymb}
\usepackage{amscd}
\usepackage{amsmath}
\usepackage{amsfonts}
\usepackage[english]{babel}
\usepackage{amssymb, amsmath, amsfonts, mathrsfs,amsthm, color}
\usepackage{amsmath}
\usepackage{amsfonts}
\usepackage{amssymb}
\usepackage{graphicx}
\usepackage{amssymb}
\usepackage{graphicx}%
\providecommand{\U}[1]{\protect\rule{.1in}{.1in}}
\newtheorem{theorem}{Theorem}[section]

\newtheorem{corollary}[theorem]{Corollary}

\newtheorem{remark}[theorem]{Remark}

\newtheorem{definition}[theorem]{Definition}
\newcommand{\R}{{\mathbb R}}
\newcommand{\N}{{\mathbb N}}

\begin{document}

\title{{Lineability and spaceability for the weak form of Peano's theorem and vector-valued sequence spaces}}
\author{C. S. Barroso\thanks{Supported by CNPq Grant 307210/2009-0.}\,\,, G. Botelho\thanks{Supported by CNPq Grant 306981/2008-4.}\,\,, V. V. F\'avaro\thanks{Supported by FAPEMIG Grant CEX-APQ-00208-09.}\,\, and D. Pellegrino\thanks{Supported by  CNPq Grant 301237/2009-3. \hfill\newline2010
}}
\date{}
\maketitle

\begin{abstract} Two new applications of a technique for spaceability are given in this paper. For the first time this technique is used in the investigation of the algebraic genericity property of the weak form of Peano's theorem on the existence of solutions of the ODE $u'=f(u)$ on $c_0$. The space of all continuous vector fields $f$ on $c_0$ is proved to contain a closed $\mathfrak{c}$-dimensional subspace formed by fields $f$ for which -- except for the null field -- the weak form of Peano's theorem fails to be true. The second application generalizes known results on the existence of closed $\mathfrak{c}$-dimensional subspaces inside certain subsets of $\ell_p(X)$-spaces, $0 < p < \infty$, to the existence of closed subspaces of maximal dimension inside such subsets.

\end{abstract}

%\vspace*{-1.0em}

\section{Introduction}
The notions of lineability and spaceability, as well as their applications, have been heavily studied by many authors in different settings lately, see, e.g., \cite{aron, ARON, LAA, pams, sinica, timoney, PAMSUCM} and references therein. The basic task in the field consists in finding linear structures, as large as possible, inside nonempty sets with certain properties. Usually, given a cardinal number $\mu$ and a (remarkable) subset $A$ of a topological vector space $E$, one wishes to show that $A \cup \{0\}$ contains a $\mu$-dimensional subspace of $E$. According to the usual terminology, $A$ is said to be:\vskip .2cm
%Let $ \mu$ be a cardinal number. A subset $A$ of a topological vector space $E$ is said to be:\\
\indent $\bullet$ {\it $\mu$-lineable} if $A \cup \{0\}$ contains a $\mu$-dimensional subspace of $E$;\vskip .2cm
\indent $\bullet$ {\it $\mu$-spaceable} if $A \cup \{0\}$ contains a closed $\mu$-dimensional subspace of $E$;\vskip .2cm
\indent $\bullet$ {\it maximal-spaceable} if it is $\mu$-spaceable with $\mu$ = dim$E$.\vskip .2cm

%Lineability and spaceability have been heavily studied in several different settings lately, see, e.g., \cite{LAA, pams, sinica} and references therein.

A standard methodology of verifying such properties consists in making a convenient manipulation of a single element of $A$ in order to define an injective linear operator $T \colon X \longrightarrow E$, where $X$ is an infinite dimensional Banach space, such that $T(X)\subseteq A \cup \{0\}$. In this case $A$ is dim$X$-lineable, and if $\overline{T(X)}\subseteq A \cup \{0\}$, then $A$ is dim$X$-spaceable. The starting element of $A$ can be called {\it mother vector}. The purpose of this paper is to discuss two new applications of the mother vector technique. First, it is explored in a situation it was never applied before (cf. Section \ref{cleon}). Thereafter, the mother vector technique is used to obtain maximal-spaceability in a framework more general than that where $\mathfrak{c}$-spaceability was obtained in \cite{LAA} (cf. Section \ref{sequences}). Next we briefly describe the results we prove by means of the mother vector technique.

Throughout this paper, $\mathfrak{c}$ will denote the cardinality of the continuum. Given a Banach space $X$, we denote by $\mathscr{K}(X)$ the set of all continuous vector fields $f \colon X \longrightarrow X$ for which the weak form of Peano's theorem, concerning the existence of local solutions of
\[
u^{\prime}(t)=f(u(t)),%~  \textcolor{green}{u(a)=b},
\]
fails to be true. For more details on this current line of research and a historical account, we refer the reader to the recent contribution of H\'ajek and Johanis \cite{Hajek-Johanis} and references therein. In Section \ref{cleon} the mother vector technique is used to prove that $\mathscr{K}(c_0)$ is $\mathfrak{c}$-spaceable in the space $C(c_0)$ of all continuous vector fields on $c_{0}$ endowed with the topology of uniform convergence on bounded subsets. We call this type of property as the algebraic genericity of differential equations in $X$. The motivation comes from studies on the generic property of differential equations in Banach spaces (see \cite{Lasota-Yorke}). Our approach to prove this result is based on Dieudonn\'e's construction of vector fields on $c_0$ failing the classical Peano's theorem (cf. \cite{Dieudonne}). To the best of our knowledge, this is the first time spaceability is studied in this context.

Our next study concerns maximal-spaceability in the setting of vector-valued sequence spaces. In \cite{LAA} the mother vector technique was used to proved that if $X$ is a Banach space, then $\ell_p(X) - \bigcup\limits_{q < p}\ell_q(X)$, $0 < p < \infty$, and  $c_0(X) - \bigcup\limits_{q > 0}\ell_q(X)$ are $\mathfrak{c}$-spaceable. In Section 3 we refine the application of the mother vector technique in this context by proving that these sets are actually maximal-spaceable. Furthermore, it is proved that, for $1 \leq p < \infty$,  $\bigcap\limits_{p < q}\ell_q(X) - \ell_p(X)$ is maximal-spaceable in the Fr\'echet space $\bigcap\limits_{p < q} \ell_q(X)$. These results are obtained as particular cases of spaceability results we prove in the more general realm of $\left(\sum_n X_n \right)_p$-spaces. As far as we know, maximal-spaceability with dimension greater than $\mathfrak{c}$ was obtained before only in \cite{sinica} for sets of non-measurable functions.

%\section{Preliminaries and basic results}\label{section2}

\section{The weak form of Peano's theorem in $c_0$}\label{cleon}

Let $X$ be a Banach space and $f\colon\mathbb{R}\times X\longrightarrow X$ be a continuous
vector field on $X$. The weak form of Peano's theorem states that if $X$ is
finite-dimensional, then the ODE
\begin{align}
\label{eqn:1}u^{\prime}=f(t,u),
\end{align}
has a solution on some open interval $I$ in $\mathbb{R}$. The study of the failure of Peano's theorem in arbitrary infinite dimensional
linear spaces was started by Dieudonn\'e \cite{Dieudonne} in 1950. He
proved the existence of a continuous vector field $f\colon  c\sb 0\longrightarrow c\sb 0$
such that if $f(t,u):=f(u)$, then the Cauchy-Peano problem associated to (\ref{eqn:1}) has no local solution around the null vector of $\mathbb{R}\times c_0$. Subsequently, counterexamples in $\ell\sb 2$, Hilbert spaces and in
nonreflexive Banach spaces were obtained by Yorke \cite{Yorke}, Godunov
\cite{Godunov2} and Cellina \cite{Cellina}, respectively. Finally, in 1973 Godunov \cite{Godunov1}
proved that Peano's theorem holds true in $X$ if and only if $X$
is finite dimensional. Further negative answers were obtained by Astala
\cite{Astala2}, Shkarin \cite{Shkarin1,Shkarin2}, Lobanov
\cite{Lobanov1} and Lobanov and Smolyanov \cite{Lobanov2} in the setting of locally convex and Fr\'echet
spaces.

Let $C(X)$ denote the linear space of all continuous vector fields on $X$, which we endow with the linear topology
of uniform convergence on bounded sets. From the lineability and spaceability point of view, the following question emerges naturally: How large is the set $\mathscr{K}(X)$ of all fields $f$ in $C(X)$ for which (\ref{eqn:1}) has no local solution?
Following the historical development of the subject, it is natural to investigate the case $X = c_0$ first. In this section we give a major step in the solution of the aforementioned question by proving that the set $\mathscr{K}(c_0)\cup\{0\}$ contains a closed $\mathfrak{c}$-dimensional subspace of $C(c_0)$.

\begin{theorem} The set of continuous vector fields on $c_0$ failing the weak form of Peano's theorem is $\mathfrak{c}$-spaceable in $C(c_0)$.\label{thm:1}
\end{theorem}

%\begin{theorem} Given $a \in \mathbb{R}$ and $b \in c_0$, let $\mathscr{K}_{(a,b)}(c\sb0)$ be the subset of $C(c\sb0)$ of all continuous vector fields on $c\sb0$ for which the Cauchy-Peano problem
%\begin{equation}u^{\prime}(t)=f(u(t)),~ u(a)=b,\label{cpp}
%\end{equation}
%has no local solution. Then
%\label{thm:1} $\mathscr{K}_{0}(c\sb 0)$ is $\bf c$-spaceable in $C(c\sb 0)$.
%\end{theorem}

%\begin{remark} By $\mathscr{K}\sb{0}(X)$ we means the set $\mathscr{K}_{(0,0)}(X)$.
%\end{remark}

%\bigskip This is the linear vector-field space version of Dieudonn\'e theorem.
%"Using Dugunji's extension theorem \cite{Dugunji} in the same way as in
%\cite{Godunov2} we immediately obtain the following corollary.

%\begin{corollary}
%There exists a closed linear subspace $E$ of $C(\ell\sb\infty)$ of dimension
%equal of continuous such that $E\subset\mathscr{K}_{0}(\ell\sb\infty)$.
%\end{corollary}

\begin{proof} Let $(e_n)_{n=1}^\infty$ be the canonical unit vectors of sequence spaces and define
%Let us consider the subset $\mathscr{K}_{0}(c\sb0)$ of $C(c\sb0)$ consisting
%of all continuous vector-fields on $c\sb0$ for which the Cauchy problem:
%\[
%u^{\prime}(t)=f(u(t)),\quad u(0)=0,
%\]
%does not have solutions.
the vector field $f\in C(c\sb0)$ by
\[
f\left(\sum\sb{n=1}^\infty x\sb ne\sb n\right)=\sum\sb{n=1}^\infty\left(\sqrt{|x\sb n|}+\frac{1}%
{n+1}\right)e\sb n.
\]
By \cite{Dieudonne} it follows that $f\in \mathscr{K}(c\sb0)$.
%In our case we shall need a slight modification of $f\sb 0$. Define $f\colon c\sb 0\to c\sb 0$ by
%\begin{equation}\label{eqn:proofThm11}
%f\Big(\sum\sb n x\sb ne\sb n\Big)=\sum\sb n\Big(\sqrt{|x\sb n|}+\frac{(-1)^n}{n+1}\Big)e\sb n.
%\end{equation}
Split $\mathbb{N}$ into countably many infinite pairwise disjoint subsets
$(\mathbb{N}_{i})_{i=1}^{\infty}$. For every $i\in\mathbb{N}$ set
$\mathbb{N}_{i}=\{i_{1}<i_{2}<\ldots\}$ and define the spreading function
$\mathbb{N}\sb if\colon c\sb0\longrightarrow c\sb0$ of $f$ over $\mathbb{N}\sb i$
by
\[
\mathbb{N}\sb if(x)=%
%TCIMACRO{\dsum \limits_{n=1}^{\infty}}%
%BeginExpansion
{\displaystyle\sum\limits_{n=1}^{\infty}}
%EndExpansion
f_{i_n}\left(  x\right)  e_{i_{n}},
\]
where $f_{n}\left(  x\right)  =\sqrt{|x\sb n|}+\frac{1}{n+1},$ for all
$n\in\mathbb{N}.$ %The conclusion of Theorem \ref{thm:1} will then follow from
%the combination of the three propositions below. \bigskip
%\begin{prop}
%The problem%
%\begin{align}
%u^{\prime}(t)  &  =f(u(t))\label{sd}\\
%u(0)  &  =0.\nonumber
%\end{align}
%has no solution in any interval containing zero.
%\end{prop}
%\begin{proof}
%If $u(t)=\left(  u_{n}(t)\right)  _{n=1}^{\infty}$ was such a solution, we would
%have%
%\begin{align*}
%u_{n}^{\prime}(t) &  =\sqrt{|u\sb n(t)|}+\frac{(-1)^n}{n+1}\\
%u_{n}(0) &  =0,
%\end{align*}
%for all $n$. In particular, choosing $n\in P=\{2k\colon k\in\mathbb{N}\}$ we would
%have (following the original argument of Dieudonn\'{e})%
%\[
%u_{n}(t)\geq\frac{t^{2}}{4}%
%\]
%for all $n\in P$ and $t>0$. This is impossible, since $(u_{n}(t))_{n=1}%
%^{\infty}\in c_{0}.$
%\end{proof}
Let us see that the map%
$$ L  \colon \ell_{1}\longrightarrow C(c_{0})~,~ L\left(  (  a_i)_{i=1}^\infty\right)    =\sum \limits_{i=1}^{\infty}
a_{i}\mathbb{N}\sb if,
$$
is well defined. Indeed, given $ (  a_i)_{i=1}^\infty \in \ell_1$ and $x\in c_{0}$, for every $m \in \mathbb{N}$ we have (the sup norm on $c_0$ is simply denoted by $\|\cdot\|$)%
\begin{align*}
\left\Vert
%TCIMACRO{\dsum \limits_{i=1}^{N}}%
%BeginExpansion
{\displaystyle\sum\limits_{i=1}^{m}}
%EndExpansion
a_{i}\mathbb{N}\sb if(x)\right\Vert    \leq%
%TCIMACRO{\dsum \limits_{i=1}^{N}}%
%BeginExpansion
{\displaystyle\sum\limits_{i=1}^{m}}
%EndExpansion
\left\vert a_{i}\right\vert \cdot \left\Vert \mathbb{N}\sb if(x)\right\Vert   \leq%
%TCIMACRO{\dsum \limits_{i=1}^{N}}%
%BeginExpansion
{\displaystyle\sum\limits_{i=1}^{m}}
%EndExpansion
\left\vert a_{i}\right\vert\cdot \left\Vert f(x)\right\Vert  =\left\Vert f(x)\right\Vert \left(
%TCIMACRO{\dsum \limits_{i=1}^{N}}%
%BeginExpansion
{\displaystyle\sum\limits_{i=1}^{m}}
%EndExpansion
\left\vert a_{i}\right\vert\right) .
\end{align*}
Making $m\longrightarrow\infty$ we conclude that $%
%TCIMACRO{\dsum \limits_{i=1}^{\infty}}%
%BeginExpansion
{\sum\limits_{i=1}^{\infty}}
%EndExpansion
a_{i}\mathbb{N}\sb if(x)\in c_{0}.$ Now let us show that $L\left(  \left(
a_{i}\right)  _{i=1}^\infty\right)  \in C(c_{0})$. For all $x,y\in c_{0}$ and
$m\in\mathbb{N}$, %
\begin{align*}
\left\Vert
%TCIMACRO{\dsum \limits_{i=1}^{N}}%
%BeginExpansion
{\displaystyle\sum\limits_{i=1}^{m}}
%EndExpansion
a_{i}\mathbb{N}\sb if(x)-%
%TCIMACRO{\dsum \limits_{i=1}^{N}}%
%BeginExpansion
{\displaystyle\sum\limits_{i=1}^{m}}
%EndExpansion
a_{i}\mathbb{N}\sb if(y)\right\Vert  &  \leq%
%TCIMACRO{\dsum \limits_{i=1}^{N}}%
%BeginExpansion
{\displaystyle\sum\limits_{i=1}^{m}}
%EndExpansion
\left\vert a_{i}\right\vert\cdot \left\Vert \mathbb{N}\sb if(x)-\mathbb{N}%
\sb if(y)\right\Vert \\
&  \leq\left\Vert f(x)-f(y)\right\Vert
%TCIMACRO{\dsum \limits_{i=1}^{N}}%
%BeginExpansion
\left(
{\displaystyle\sum\limits_{i=1}^{m}}
%EndExpansion
\left\vert a_{i}\right\vert \right).
\end{align*}
Then by taking the limit as $m\longrightarrow\infty$, the continuity of $%
%TCIMACRO{\dsum \limits_{i=1}^{\infty}}%
%BeginExpansion
{\sum\limits_{i=1}^{\infty}}
%EndExpansion
a_{i}\mathbb{N}\sb if$ follows from the continuity of $f$. Since $L$  is well defined, its linearity and injectivity are clear, so the range space
$L(\ell_{1})$ is algebraically isomorphic to $\ell_{1}$. We claim that%
\begin{equation}\label{jj}
\overline{L(\ell_{1})}\subseteq\mathscr{K}\sb(c\sb 0)\cup\{0\}.
\end{equation}
Let $h=(h_{i})_{i=1}^\infty\in\overline{L(\ell\sb1)}$ be arbitrary. We may assume that $h\neq 0$, so there is $r\in\mathbb{N}$ such
that $h_{r}\neq0.$ Using the decomposition $\mathbb{N}=\bigcup\limits
_{j=1}^{\infty}\mathbb{N}_{j}$ there are (unique) $m,s\in\mathbb{N}$ such that $e_{m_{s}}=e_{r}$. Let $(x\sb
k)\sb {k=1}^\infty =\left((a\sb i^{k})\sb {i=1}^\infty\right)\sb {k=1}^\infty$ be a sequence in $\ell\sb1$ so that
\[
L\left(x\sb k\right)=\sum\sb{j=1}^{\infty}a\sb j^{k}\mathbb{N}\sb jf\stackrel{k \rightarrow \infty}{\longrightarrow}
h\text{ in }C(c\sb0).
\]
%It is easy to see that
%\[
%L(x\sb k)=\sum\sb{i}a\sb i^{k}\big(\mathbb{N}\sb if\big)\sb ie\sb i.
%\]
Letting $L\sb n(x\sb k)$ denote the $n$-th coordinate of $L(x\sb k)$, for
each $N\in\mathbb{N}$ we have that
\[h_n = \lim\limits_{k \rightarrow \infty} L\sb n(x\sb k)
\]
uniformly in the ball $B\sb{c\sb 0}(N):=\{x \in c_0 : \|x\| \leq N\}$. Since $
L_{i_{j}}(x\sb k)=a_{i}^{k}f_{i_j}$ for all $i,j\in\mathbb{N}$, it follows that
\[
a_{i}^{k}f_{i_j}(x)=L_{i_{j}}(x\sb k)(x)\stackrel{k \rightarrow \infty}{\longrightarrow}
h_{i_j}(x)
\]
for each $x\in B\sb{c\sb 0}(N)$.
In particular,
\begin{equation} \label{refnova}a_{m}^{k}f_{m_j}(x)=L_{m_j}(x\sb k)(x)\stackrel{k \rightarrow \infty}{\longrightarrow}
h_{m_j}(x)
\end{equation}
for every $j$ and every $x\in B\sb{c\sb 0}(N)$; and making $j = s$ we get
\begin{equation} \label{novanova}a_{m}^{k}f_{r}(x)= a_{m}^{k}f_{m_s}(x)\stackrel{k \rightarrow \infty}{\longrightarrow}
h_{m_s}(x)= h_{r}(x)
\end{equation}
for each $x\in B\sb{c\sb 0}(N)$. Choosing $x_0\in c_0$ such that $h_r(x_0) \neq 0$ and $N_0 \in \mathbb{N}$ such that $x_0 \in B\sb{c\sb 0}(N_0)$ it follows that
\[
a_r := \lim\limits_{k \rightarrow \infty}a_{m}^{k}= \frac{h\sb r(x_0)}{f_{r}(x_0)}\neq 0.%
\]
Thus (\ref{novanova}) implies that $h\sb r(x)=a\sb rf_{r}(x)$ for all
$x\in B\sb{c\sb 0}(N)$. As $N$ is arbitrary, we have $h\sb r(x)=a\sb rf_{r}(x)$ for every $x\in c\sb0$.
Since for every $j,k\in \mathbb{N}$ the $m_{j}$-th coordinate of $L\left(x\sb k\right)$ is $a_{m}^{k}f_{m_j}$, by (\ref{refnova}) we have that $h\sb{m_{j}}(x)=a\sb rf_{m_j}(x)$, for all $j\in \mathbb{N}$. %In other words, $$\left(h_{m_j}(x)\right)_{j=1}^\infty = a_r \mathbb{N}_mf(x) {\rm ~~ for~ every~} x \in c_0.$$ %Hence
%\[
%h(x)=a\sb r\mathbb{N}\sb mf(x) + \textrm{alguma coisa}.
%\]
%Agora eu acho que d\'a para acabar a demonstra\c c\~ao com os argumentos deles, %pois basta aparecer um $\mathbb{N}\sb mf$ para dar o absurdo, correto?
%Using this argument
%for each $r\in\mathbb{N}$ and denoting by $s\left(  r\right)  $ the corresponding
%$s$ such that $e_{r}=e_{m_{s}},$ we obtain%
%\[
%h(x)=\sum\sb{r=1}^{\infty}a_{r}f_{s\left(  r\right)  }(x),
%\]
%for all $x\in c\sb0,$ where%
%\[
%a_{r}=\left\{
%\begin{array}
%[c]{c}%
%\frac{h\sb r(x)}{f_{s}(x)},\text{ if }h_{r}(x)\neq0\\
%0,\text{ if }h_{r}=0
%\end{array}
%\right.  .
%\]
%We shall use the representation ?? of $h$ to prove that $h \in \mathscr{K}\sb0(c\sb0)$.

Now we are ready to prove that $h \in \mathscr{K}(c\sb0)$. We proceed by contradiction using the ODE
approach from \cite{Dieudonne}. Assume that $u(t)=\left(  u_{n}(t)\right)  _{n=1}^\infty$ is a solution of (\ref{eqn:1}) on some interval $I\subset \mathbb{R}$. Fix any $a\in I$ and write $b=u(a)=(b_i)_{i=1}^\infty$. In this case we have
\[
u_{m_j}^{\prime}(t) =h_{m_j}(u(t)) = a_rf_{m_j}(u(t)) = a_r\left(\sqrt{|u_{m_j}(t)|}+\frac{1}{m_j+1}\right),\]
and $u_{m_j}(a)=b_{m_j}$ for all $j\in \mathbb{N}$ and $t \in I$. In summary, each $u_{m_j}$ is a solution of the Cauchy problem
\begin{equation}\label{eqnova}
u_{m_j}^{\prime}(t) = a_r\left(\sqrt{|u_{m_j}(t)|}+\frac{1}{m_j+1}\right),~ u_{m_j}(a)=b_{m_j},
\end{equation}
for all $j\in\mathbb{N}$ and all $t\in I$, and hence for all $t\in\mathbb{R}$.
Let us recall the original argument of Dieudonn\'{e} \cite{Dieudonne}: if $\alpha,\beta\in\mathbb{R}$ and $\gamma>0$ then
\[\int_\alpha^\beta\frac{dx}{\sqrt{|x|}+\gamma}\leq 2(\sqrt{|\alpha|}+|\sqrt{\beta|}).
\]
Thus if $u'(t)=\lambda\left(\sqrt{|u(t)|}+\gamma\right)$ with $\gamma,\lambda>0$,  $t\geq t_0$ and $u(t_0)=y_0$, then
\begin{align*}
t-t_0=\int_{t_0}^t\frac{u'(s)ds}{\lambda\left(\sqrt{|u(s)|}+\gamma\right)}= \frac{1}{\lambda}\int_{u(t_0)}^{u(t)}\frac{dx}{\sqrt{|x|}+\gamma}
\leq \frac{2}{\lambda}\left(\sqrt{|u(t)|}+\sqrt{|u(t_0)|}\right) .
\end{align*}
In view of (\ref{eqnova}), if $a_r>0$ then $$0< \frac{a_r(t-a)}{2}\leq \sqrt{|u_{m_j}(t)|} + \sqrt{|u_{m_j}(a)|}$$ for all $t>a$ and all $j\in\mathbb{N}$. This contradiction -- remember that $(u_{m_j}(t))_{j\in\mathbb{N}}\in c_{0}$ for $t\in I$ -- shows that $h\in \mathscr{K}(c\sb0)$.
If $a_r<0$, we can define $v_{m_j}(t)=u_{m_j}(-t)$ for all
$t\in\mathbb{R}$ and $j\in\mathbb{N}$. Applying once more Dieudonn\'e's argument  we get %
\[0<
\frac{-a_r(t-a)}{2}\leq \sqrt{|u_{m_j}(-t)|} + \sqrt{|u_{m_j}(-a)|} \]for all $ t>a$ and  all $j\in\mathbb{N}$;
which is impossible since $\left(  u_{m_j}(-t)\right)  _{j\in\mathbb{N}}\in c_{0}$ for $t \in I$. Hence $h\in \mathscr{K}(c\sb0)$.

So (\ref{jj}) is established, proving that $\overline{L(\ell_1)}$ is a closed $\mathfrak{c}$-dimensional subspace of $C(c_0)$ contained in $\mathscr{K}(c\sb0)\cup\{0\}$. %By $[A]$ we mean the subspace generated by the subset $A$ of a linear space. Next we prove that
\end{proof}

%\begin{remark}\rm \textcolor{red}{Falar que de fato estamos conseguindo espa\c cabilidade maximal, isto \'e, a dimens\~ao de $C(c_0)$ \'e $\bf c$. (Esperando Vin\'icius).}
%\end{remark}

\section{Vector-valued sequence spaces}\label{sequences}
Let $(X_n)_{n=1}^\infty$ be a sequence of Banach spaces over $\mathbb{K} = \mathbb{R}$ or $\mathbb{C}$. Given $0 < p < \infty$, by $\left(\sum_n X_n \right)_p$ we mean the vector space of all sequences $(x_n)_{n=1}^\infty$ such that $x_n \in X_n$ for every $n$ and
$$\|(x_n)_{n=1}^\infty\|_p := \left(\sum_{n=1}^\infty \|x_n\|_{X_n}^p\right) < \infty. $$
It is well known that $\left(\left(\sum_n X_n \right)_p, \|\cdot\|_p\right)$ is a Banach ($p$-Banach if $0 < p <1$) space. Making the obvious modification for $p = 0$ we get the Banach space $\left(\sum_n X_n \right)_0$ of norm null sequences with the sup norm. In this fashion,
$$\left(\textstyle\sum_n X_n \right)_p^- := \textstyle\bigcup\limits_{0 < q < p} (\sum_n X_n)_q $$
can be regarded as a subspace of $(\sum_n X_n)_p$ and $\bigcup\limits
_{p>0}\left(\sum_n X_n \right)_p$ can be regarded as a subspace of $\left(\sum_n X_n \right)_0$.

For a Banach space $X$, the usual $X$-valued sequence spaces $\ell_p(X)$ and $c_0(X)$, as well as their corresponding subspaces $\ell_p^-(X):= \bigcup\limits
_{0<q<p}\ell_{q}(X)$ and $\bigcup\limits
_{p>0}\ell_{p}(X)$, are recovered putting $X_n = X$ for every $n$. In \cite{LAA} it is proved that $\ell_p(X) - \ell_p^-(X)$ and $c_0(X) - \bigcup\limits
_{p>0}\ell_{p}(X)$ are $\mathfrak{c}$-spaceable. In this section we shall prove that these sets are actually maximal-spaceable (cf. Corollary \ref{cor}). This information will be obtained as a particular case of spaceability results in the more general realm of $\left(\textstyle\sum_n X_n \right)_p$-spaces.

\begin{definition}\rm Given a Banach space $X$, a family $(X_i)_{i \in I}$ of Banach spaces is said to {\it contain isomorphs of $X$ uniformly} if there are $\delta > 0$ and a family of isomorphisms into $R_i \colon X \longrightarrow X_i$ such that $\min\{ \|R_i\|,\|R_i^{-1}\|\} \leq \delta$ for every $i \in I$.
\end{definition}

\begin{theorem}\label{theo3}
Let $(X_n)_{n=1}^\infty$ be a sequence of Banach spaces that contains a subsequence containing isomorphs of the infinite dimensional Banach space $X$ uniformly. Then:\\
{\rm (a)} $\left(\sum_n X_n \right)_p- \left(\sum_n X_n \right)_p^-$ is {\rm dim}$X$-spaceable for every $0 < p < \infty$.\\
{\rm (b)} $\left(\sum_n X_n \right)_0 - \bigcup\limits
_{p>0}\left(\sum_n X_n \right)_p$ is {\rm dim}$X$-spaceable.%\\
%{\rm (c)} $(\sum_n X_n)_p^+ - (\sum_n X_n)_p$ is maximal-spaceable for every $1 \leq p < \infty$.
\end{theorem}

%Given a Banach space $X$ and $0 < p < \infty$, regard $\ell_p^-(X):= \bigcup\limits
%_{0<q<p}\ell_{q}(X)$ as a subspace of the Banach ($p$-Banach if $0 < p < 1)$ space $\ell_p(X)$. In the same fashion, $\bigcup\limits
%_{p>0}\ell_{p}(X)$ can be regarded as a subspace of the Banach space $c_0(X)$. In \cite{LAA} it is proved that $\ell_p(X) - \ell_p^-(X)$ and $c_0(X) - \bigcup\limits
%_{p>0}\ell_{p}(X)$ are $\bf c$-spaceable. In this section we prove that these sets are actually maximal-spaceable. This information will be obtained as a particular case of a more general result.
%
%\begin{theorem}\label{theo}
%Let $X$ be an infinite dimensional Banach space. Then:\\
%{\rm (a)} $\ell_{p}(X)- \ell_p^-(X)$ is maximal-spaceable for every $0 < p < \infty$.\\
%{\rm (b)} $c_0(X) - \bigcup\limits
%_{p>0}\ell_{p}(X)$ is maximal-spaceable.
%\end{theorem}

\begin{proof} It is plain that we can assume, without loss of generality, that the sequence $(X_n)_{n=1}^\infty$ contain isomorphs of $X$ uniformly. So there are $\delta > 0$ and isomorphisms into $R_n \colon X \longrightarrow X_n$ such that $\|R_n\| \leq \delta$ and $\|R_n^{-1}\| \leq \delta$ for every $n \in \N$.

\medskip

\noindent(a)
Let $\xi=(\xi_{j})_{j=1}^\infty\in\ell_{p}-\bigcup\limits_{0<q<p}\ell_{q}$. Split $\mathbb{N}$
into countably many infinite pairwise disjoint subsets $(\mathbb{N}_{i}%
)_{i=1}^{\infty}$. For every $i\in\mathbb{N}$ set $\mathbb{N}_{i}%
=\{i_{1}<i_{2}<\ldots\}$ and, denoting by $(e_n)_{n=1}^\infty$ the canonical unit vectors of sequence spaces, define
\[
y_{i}=\sum_{j=1}^{\infty}\xi_{j}e_{i_{j}}\in \mathbb{K}^{\mathbb{N}}.
\]
Since $\Vert y_{i}\Vert_{r}=\Vert\xi\Vert_{r}$ for every $r>0$, we have that
$y_{i}\in\ell_{p}-\bigcup\limits_{0<q<p}\ell_{q},$ for every $i$. For $x = (x_n)_{n=1}^\infty \in \mathbb{K}^\mathbb{N}$ and $w \in X$ we write
$$x \otimes w := (x_nR_n(w))_{n=1}^\infty \in \textstyle\prod\limits_{n=1}^\infty X_n.$$
It is clear that, for all $w,w_1,w_2 \in X$ and $\lambda \in \mathbb{K}$,
\begin{equation}\label{tensor}x \otimes (w_1 + w_2) = x \otimes w_1 + x \otimes w_2 {\rm ~~and~~} \lambda(x \otimes w) = (\lambda x) \otimes w = x \otimes (\lambda x),
\end{equation}
what justifies the use of the symbol $\otimes$.  Define $\tilde
{s}=1$ if $p\geq1$ and $\tilde{s}=p$ if $0<p<1$. Let $(w_{j})_{j=1}^{\infty
}\in\ell_{\tilde{s}}(X)$ be given. As $y_j \otimes w_j \in \prod\limits_{n=1}^\infty X_n$ for every $j \in \mathbb{N}$, and
\begin{align*}\|y_j \otimes w_j\|_p &= \left(\sum_{k=1}^\infty |\xi_k|^p \cdot \| R_{j_k}(w_j)\|_{X_{j_k}}^p \right)^{\frac1p} \leq \left(\sum_{k=1}^\infty |\xi_k|^p \cdot \| R_{j_k}\|^p \cdot \|w_j\|_X^p \right)^{\frac1p}\\
& \leq \delta \|w_j\|_X \cdot \|\xi\|_{p} < \infty,
\end{align*}
each $y_j \otimes w_j \in \left(\sum_n X_n \right)_p$. Moreover,
\begin{align*}
\sum_{j=1}^{\infty}\Vert y_{j} \otimes w_j\Vert_{p}^{\tilde{s}} &  \leq \sum
_{j=1}^{\infty} \left(\delta \|w_j\|_X \cdot \|\xi\|_{p}\right)^{\tilde s}= \delta^{\tilde s}\|\xi\|_{p}^{\tilde s}\cdot \|(w_j)_{j=1}^\infty\|_{\tilde s}^{\tilde s} < \infty.%  \\&=\sum
%_{j=1}^{\infty}\left\Vert\sum_{k=1}^{\infty}w_{j}\xi_{k}e_{j_{k}}\right\Vert_{p}%
%^{\tilde{s}}=\sum_{j=1}^{\infty}\left(  \sum_{k=1}^{\infty}\left\Vert w_{j}%
%\xi_{k}\right\Vert _{X}^{p}\right)  ^{\frac{\tilde{s}}{p}}\\
%&  =\sum_{j=1}^{\infty}\left(  \sum_{k=1}^{\infty}\left\Vert w_{j}\right\Vert
%_{X}^{p}\cdot \left\vert \xi_{k}\right\vert ^{p}\right)  ^{\frac{\tilde{s}}{p}}%
%=\sum_{j=1}^{\infty}\left\Vert w_{j}\right\Vert _{X}^{\tilde{s}}\cdot \left(
%\sum_{k=1}^{\infty}\left\vert \xi_{k}\right\vert ^{p}\right)  ^{\frac
%{\tilde{s}}{p}}\\
%&  =\sum_{j=1}^{\infty}\left\Vert w_{j}\right\Vert _{X}^{\tilde{s}}\cdot \Vert
%\xi\Vert_{p}^{\tilde{s}}=\Vert\xi\Vert_{p}^{\tilde{s}}\cdot \left\Vert (w_{j}%
%)_{j=1}^{\infty}\right\Vert _{\tilde{s}}<\infty.
\end{align*}
Thus $\sum\limits_{j=1}^{\infty}\Vert y_{j}\otimes w_j\Vert_{p}<\infty$ if $p\geq1$ and
$\sum\limits_{j=1}^{\infty}\Vert y_{j}\otimes w_j\Vert_{p}^{p}<\infty$ if $0<p<1$. Hence
the series $\sum\limits_{j=1}^{\infty}y_{j}\otimes w_j$ converges in $\left(\sum_n X_n \right)_p$ and the
operator
\[
T\colon\ell_{\tilde{s}}(X)\longrightarrow \left(\textstyle\sum_n X_n \right)_p~~,~~T\left(
(w_{j})_{j=1}^{\infty}\right)  =\sum\limits_{j=1}^{\infty}y_{j}\otimes w_j,%
\]
is well defined. From (\ref{tensor}) it follows easily that $T$ is linear, and from the fact that the sequences $(y_j)_{j=1}^\infty$ are disjointly supported it follows that $T$ is injective. Thus
$\overline{T\left(  \ell_{\tilde{s}}(X)\right)  }$ is a closed
infinite dimensional subspace of $\left(\sum_n X_n \right)_p$ and $$\text{dim}\overline
{T\left(  \ell_{\tilde{s}}(X)\right)  }=\text{dim}\ell_{\tilde{s}%
}(X)=\text{dim} X.$$ Now we just have to show that $$\overline{T\left(
\ell_{\tilde{s}}(X)\right)  }-\left\{  0\right\}  \subseteq\left(\textstyle\sum_n X_n \right)_p-\textstyle\bigcup\limits_{0<q<p}\left(\sum_n X_n \right)_q.$$ Let $z=\left(  z_{n}\right)  _{n=1}^{\infty
}\in\overline{T\left(  \ell_{\tilde{s}}(X)\right)  },$ $z\neq0$. There are
sequences $\left(  w_{i}^{(k)}\right)  _{i=1}^{\infty}\in\ell_{\tilde{s}}(X)$,
$k\in\mathbb{N}$, such that $z=\lim\limits_{k\rightarrow\infty}T\left(  \left(
w_{i}^{(k)}\right)  _{i=1}^{\infty}\right)  $ in $\left(\sum_n X_n \right)_p.$ Note that, for
each $k\in\mathbb{N}$,%
\[
T\left(  \left(  w_{i}^{(k)}\right)  _{i=1}^{\infty}\right)  =\sum
\limits_{i=1}^{\infty}y_i \otimes w_{i}^{(k)}=\sum\limits_{i=1}^{\infty}\left(\sum_{j=1}^{\infty}\xi_{j}e_{i_{j}}\right) \otimes w_{i}%
^{(k)}=\sum\limits_{i=1}^{\infty}%
\sum\limits_{j=1}^{\infty}\xi_{j}e_{i_{j}}\otimes w_{i}^{(k)},
\]
what means that, for every $i,j \in \mathbb{N}$, the $i_j$-th coordinate of $T\left(  \left(  w_{i}^{(k)}\right)  _{i=1}^{\infty}\right)$ is $\xi_j R_{i_j}(w_i^{(k)})$. Fix $r\in\mathbb{N}$ such that $z_{r}\neq0.$ Since $\mathbb{N}=\bigcup\limits
_{j=1}^{\infty}\mathbb{N}_{j}$, there are (unique) $m,t\in\mathbb{N}$ such
that $m_t=r$. Thus, for each $k\in\mathbb{N}$, the $r$-th coordinate
of $T\left(  \left(  w_{i}^{(k)}\right)  _{i=1}^{\infty}\right)  $ is
$\xi_{t}R_r(w_{m}^{(k)}).$ Since convergence in $\left(\sum_n X_n \right)_p$ implies
coordinatewise convergence, we have $z_{r}=\lim\limits_{k\rightarrow\infty}\xi_{t}R_r(w_{m}^{(k)})$. From $z_r \neq 0$ it follows that $\xi_t \neq 0$ and
\[
z_{r}=\lim_{k\rightarrow\infty}\xi_{t}R_r(w_{m}^{(k)})=\xi_{t}\cdot\lim
_{k\rightarrow\infty}R_r(w_{m}^{(k)}) \in \overline{R_r(X)} = R_r(X).
\]
Hence $R_r^{-1}(z_r) = \xi_t\cdot \lim\limits_{k \rightarrow \infty} w_m^{(k)}$. Call $$\alpha_m :=\lim\limits_{k\rightarrow\infty}w_{m}^{(k)}=\frac{R_r^{-1}(z_{r})%
}{\xi_{t}}\neq0.$$ For $j,k\in\mathbb{N}$, the $m_{j}$-th coordinate of
$T\left(  \left(  w_{i}^{(k)}\right)  _{i=1}^{\infty}\right)  $ is
$\xi_{j}R_{m_j}(w_{m}^{(k)}).$ On the one hand,
\[
\lim_{k\rightarrow\infty}\xi_{j}R_{m_j}(w_{m}^{(k)})=\xi_{j}\cdot\lim_{k\rightarrow
\infty}R_{m_j}(w_{m}^{(k)})=\xi_{j} R_{m_j}\left(\lim_{k\rightarrow
\infty}w_{m}^{(k)}\right)= \xi_j R_{m_j}(\alpha_{m})%
\]
for every $j\in\mathbb{N}$. On the other hand, coordinatewise convergence
gives $$\lim\limits_{k\rightarrow\infty}\xi_{j}R_{m_j}(w_{m}^{(k)})=z_{m_{j}},$$ so $z_{m_{j}%
}=\xi_i R_{m_j}(\alpha_{m})$ for each $j\in\mathbb{N}$.\newline Finally, for all $0<q<p$,
\begin{align*}
\left\Vert z\right\Vert _{q}^{q}&=\sum\limits_{n=1}^{\infty}\left\Vert
z_{n}\right\Vert _{X_n}^{q}\geq\sum\limits_{j=1}^{\infty}\left\Vert z_{m_{j}%
}\right\Vert _{X_{m_j}}^{q}=\sum\limits_{j=1}^{\infty}\left\vert \xi_{j}\right\vert ^{q}\cdot\left\Vert R_{m_j}(\alpha
_{m})\right\Vert _{X_{m_j}}^{q}\\&\geq \sum\limits_{j=1}^{\infty}\left\vert \xi_{j}\right\vert ^{q}\cdot\frac{1}{\|R_{m_j}^{-1}\|^q}\cdot\left\Vert \alpha
_{m}\right\Vert _{X}^{q} \geq \frac{1}{\delta^q}\left\Vert
\alpha_{m}\right\Vert _{X}^{q}\cdot\left\Vert \xi\right\Vert _{q}^{q}=\infty,
\end{align*}
proving that $z\notin\bigcup\limits_{0<q<p}\left(\sum_n X_n \right)_q$.

\medskip

\noindent (b) Start with a sequence $\xi=(\xi_{j})\in c_{0}-\bigcup\limits_{p>0}\ell_{p}$ and proceed as before to define the operator
\[
T\colon\ell_{1}(X)\longrightarrow \left(\textstyle\sum_n X_n\right)_0~~,~~T\left(  (w_{j})_{j=1}^{\infty
}\right)  =\sum\limits_{j=1}^{\infty}y_j \otimes w_{j}.%
\]
The well-definiteness of $T$ is even easier in this case. Again $T$ is linear, injective and the same steps of the proof of (a) show that
$$\overline{T\left(  \ell
_{1}(X)\right)  }-\left\{  0\right\}  \subseteq  \left(\textstyle\sum_n X_n\right)_0 -\textstyle\bigcup\limits_{p>0}\left(\textstyle\sum_n X_n\right)_p.$$
As dim$\overline{T\left(  \ell
_{1}(X)\right)  } = {\rm dim}X$, the proof is complete.
\end{proof}

Let $(X_n)_{n=1}^\infty$ be a sequence of Banach spaces and $1\leq p<+\infty.$ We define%
$$(\textstyle\sum_n X_n)_p^+ := \textstyle\bigcap\limits_{q>p}\left(\sum_n X_n \right)_q = \textstyle\bigcap\limits_{k \in \mathbb{N}}\left(\sum_n X_n \right)_{p_k}, $$
%\[
%\left(\sum_n X_n)\right)_p^+\ell_{p}^{+}\left(  X\right)  %=\textstyle\bigcap\limits_{q>p}\ell_{q}\left(  X\right)
%=\textstyle\bigcap\limits_{k\in\mathbb{N}}\ell_{p_{k}}\left(  X\right)  ,
%\]
where $\left(  p_{k}\right)  _{k=1}^{\infty}$ is any decreasing sequence
converging to $p,$ endowed with the locally convex topology $\tau$ generated by the family of norms%
\[
\left\Vert \left(  x_{n}\right)  _{n=1}^{\infty}\right\Vert _{q}=\left(
\sum_{n=1}^{\infty}\left\Vert x_{n}\right\Vert _{X_n}^{q}\right)  ^{\frac{1}{q}%
}, ~q>p.\]
This locally convex topology $\tau$ is clearly generated
by the countably family of norms%
\[
\left\Vert \left(  x_{n}\right)  _{n=1}^{\infty}\right\Vert _{p_{k}}=\left(
\sum_{n=1}^{\infty}\left\Vert x_{n}\right\Vert _{X_n}^{p_{k}}\right)  ^{\frac
{1}{p_{k}}},~k \in \mathbb{N},
\]
so $((\textstyle\sum_n X_n)_p^+, \tau)  $ is metrizable. The completeness can be proved similarly to the case of $(\textstyle\sum_n X_n)_p $. Alternatively, note that $\tau$ is the projective limit topology defined by the inclusions $(\textstyle\sum_n X_n)_p^+ \hookrightarrow(\textstyle\sum_n X_n)_q  ,$ $q>p.$ So it is complete as the projective limit of complete Hausdorff spaces. In summary $((\textstyle\sum_n X_n)_p^+, \tau)  $ is a Fr\'{e}chet space.

If $X_n = X$ for every $n$, we write $\ell_p(X)^+$. In particular, for $X=\mathbb{K}$, $\ell_{p}^{+}:=\ell_{p}^{+}\left(  \mathbb{K}\right)  $
recovers the space $l^{p+}$ introduced by Metafune and Moscatelli
\cite{metafune}.

\begin{remark}\rm It is easy to see that if $X_n \neq \{0\}$ for every $n$, then the inclusions
$$(\textstyle\sum_n X_n)_p\subseteq (\textstyle\sum_n X_n)_p^+
\subseteq (\textstyle\sum_n X_n)_q$$ are strict whenever $1 \leq   p < q$.% $q>p$, are proper whenever $X\neq\{0\}$. In fact, for the first
%inclusion, letting $x\in X,$ $x\neq0$, then$\left(  \frac{1}{n^{p}}x\right)
%_{n=1}^{\infty}\in\ell_{p}^{+}\left(  X\right)  -\ell_{p}(X).$ For the second
%inclusion, it is enough to observe that $\ell_{p}^{+}\left(  X\right)
%\subseteq\ell_{r}(X)\subseteq\ell_{q}(X)$ but $\ell_{r}(X)\neq\ell_{q}(X)$ if
%$p<r<q.$
\end{remark}

\begin{theorem}\label{theo2}
Let $(X_n)_{n=1}^\infty$ be a sequence of Banach spaces that contains a subsequence containing isomorphs of the infinite dimensional Banach space $X$ uniformly.  Then $(\textstyle\sum_n X_n)_p^{+}  -(\textstyle\sum_n X_n)_p$ is {\rm dim}$X$-spaceable.
\end{theorem}

\begin{proof} The proof goes along the same steps of the proof of Theorem \ref{theo3}. We sketch the argument highlighting the differences the locally convex topology of $\left(\sum_n X_n \right)_p^+$ arises. In this proof, {\it as before} means {\it as in the proof of Theorem \ref{theo3}}.

Let $\xi=(\xi_{j})\in\ell_{p}^{+}-\ell_{p}$. Write $\N = \textstyle\bigcup\limits_{i=1}^\infty \N_i$, $\mathbb{N}_{i}=\{i_{1}<i_{2}<\ldots\}$, and define $y_j \in \mathbb{K}^{\N}$, $j \in \N$, as before. Since $\Vert y_{i}\Vert_{r}=\Vert\xi\Vert_{r}$ for every $r>0$, we have that
$y_{i}\in\ell_{p}^{+}-\ell_{p},$ for every $i$. Given $(w_{j})_{j=1}^{\infty}%
\in\ell_{1}(X)$, let us show that the series $\sum\limits_{j=1}^{\infty}y_j \otimes w_{j}$ converges in $(\textstyle\sum_n X_n)_p^{+}$, where $y_j \otimes w_j$ is defined as before. Write $s_n = \sum\limits_{j=1}^{n}y_j \otimes w_{j}$, $n \in \mathbb{N}$. For a fixed $q>p,$ the same computations we performed before show that each $y_j \otimes w_j \in (\textstyle\sum_n X_n)_q$ and
$$\sum_{j=1}^{n}\Vert y_{j} \otimes w_j\Vert_{q}\leq \delta \Vert\xi\Vert_{q}\cdot \left\Vert (w_{j}%
)_{j=1}^{\infty}\right\Vert _{1}
$$ for every $n$. As $ (\textstyle\sum_n X_n)_q$ is a Banach space, there is $S_q \in (\textstyle\sum_n X_n)_q$ such that $S_q = \lim\limits_{n \rightarrow \infty} s_n$ in $ (\textstyle\sum_n X_n)_q$. If $q,q' > p$, say $q \leq q'$, then $S_{q} \in  (\textstyle\sum_n X_n)_{q'}$ and
 $$\|s_n - S_{q}\|_{q'} \leq \| s_n - S_{q}\|_{q} \longrightarrow 0, $$
 showing that $S_q = \lim\limits_{n \rightarrow \infty} s_n$ in $ (\textstyle\sum_n X_n)_{q'}$, therefore $S_q = S_{q'}$. This shows that $S_q$ does not depend on $q$, so there is $S \in  (\textstyle\sum_n X_n)_q$ such that $s_n \longrightarrow S$ in $ (\textstyle\sum_n X_n)_q$ for every $q > p$. Hence $S \in  (\textstyle\sum_n X_n)_q^+$ and $s_n \longrightarrow S$ in the topology of $ (\textstyle\sum_n X_n)_p^+$. In other words, $\sum\limits_{j=1}^{\infty}y_{j}\otimes w_j\in (\textstyle\sum_n X_n)_p^+$ and the operator%
\[
T\colon\ell_{1}(X)\longrightarrow  (\textstyle\sum_n X_n)_p^+  ~~,~~T\left(
(w_{j})_{j=1}^{\infty}\right)  =\sum\limits_{j=1}^{\infty}y_{j}\otimes w_j%
\]
is then well defined. As before, $T$ is linear and injective. Thus
$\overline{T\left(  \ell_{1}(X)\right)  }$ is a closed dim$X$-dimensional
subspace of $ (\textstyle\sum_n X_n)_p^+$. % and
%\[
%\text{dim}\overline{T\left(  \ell_{1}(X)\right)  }%
%= \text{dim}\ell_{1}(X)=\text{dim}X=\lambda.
%\]
Now we just have to show that $$\overline{T\left(  \ell_{1}(X)\right)
}-\left\{  0\right\}  \subseteq  (\textstyle\sum_n X_n)_p^+  -  (\textstyle\sum_n X_n)_p.$$
Let $z=\left(  z_{n}\right)  _{n=1}^{\infty}\in\overline{T\left(  \ell
_{1}(X)\right)  },$ $z\neq0$. There are sequences $w_k =\left( w_{i}^{(k)}\right)
_{i=1}^{\infty}\in\ell_{1}(X)$, $k\in\mathbb{N}$, such that $z=\lim\limits_{k\rightarrow\infty}T\left(  \left(
w_{i}^{(k)}\right)  _{i=1}^{\infty}\right)  $ in $\left(\sum_n X_n \right)_p^+$. Since the topology of $\left(\sum_n X_n \right)_p^+$ is generated by the norms $\| \cdot\|_q$, $q > p$, it follows that
\[\lim_{k \rightarrow \infty}\left\Vert
T\left(  \left(  w_{i}^{(k)}\right)  _{i=1}^{\infty}\right)  -z\right\Vert
_{q}= 0 {\rm~for~every}~ q>p.\]
Fix $q > p$ and use coordinatewise convergence in $\left(\sum_n X_n \right)_q$ as before to conclude that
 $z\notin\left(\sum_n X_n \right)_p$.
\end{proof}

Making $X_n = X$ for every $n$ in Theorems \ref{theo3} and \ref{theo2}, since dim$\ell_p(X)$ = dim$c_0(X)$ = dim$\ell_p(X)^+$= dim$X$, we get:

\begin{corollary}\label{cor}
Let $X$ be an infinite dimensional Banach space $X$. Then:\\
{\rm (a)} $\ell_p(X)- \ell_p(X)^-$ is maximal-spaceable for every $0 < p < \infty$.\\
{\rm (b)} $c_0(X) - \bigcup\limits
_{p>0}\ell_p(X)$ is maximal-spaceable.\\
{\rm (c)} $\ell_p(X)^+ - \ell_p(X)$ is maximal-spaceable for every $1 \leq p < \infty$.
\end{corollary}

%\bigskip

Of course Theorems \ref{theo3} and \ref{theo2} apply to many other interesting situations, but we refrain from going into details. %We list just a few:

\begin{remark}\rm Note that in the proofs of Theorems \ref{theo3} and \ref{theo2} we start with a vector $\xi$ belonging to a set other than the one we prove to be spaceable. It is this variation of the mother vector technique that allows us obtain maximal-spaceability rather than $\mathfrak{c}$-spaceability in Corollary \ref{cor}.
\end{remark}

\bigskip

\noindent {\bf Acknowledgement.} The authors thank J. M. Ansemil for drawing our attention to the spaces $\ell_p^+$ and for pointing out reference \cite{metafune}.

\vspace*{1em}

\noindent Cleon S. Barroso\\
Departamento de Matem\'atica, Campus do Pici\\
Universidade Federal do Cear\'a\\
60.455-760 -- Fortaleza -- Brazil\\
e-mail: cleonbar@mat.ufc.br

\medskip

\noindent Geraldo Botelho and Vin\'icius V. F\'avaro\\
Faculdade de Matem\'atica\\
Universidade Federal de Uberl\^andia\\
38.400-902 -- Uberl\^andia -- Brazil\\
e-mails: botelho@ufu.br, vvfavaro@gmail.com

\medskip

\noindent Daniel Pellegrino\\
Departamento de Matem\'atica\\
Universidade Federal da Para\'iba\\
58.051-900 -- Jo\~ao Pessoa -- Brazil\\
e-mail: pellegrino@pq.cnpq.br and dmpellegrino@gmail.com


\begin{thebibliography}{99}\small

\bibitem{aron} R. M. Aron, F. J. Garc\'{\i}a-Pacheco, D. P\'erez-Garc\'{\i}a, J. B. Seoane-Sep\'ulveda, {\it On dense-lineability of sets of functions on $\mathbb{R}$}, Topology {\bf 48} (2009), 149--156.

\vspace{-0.5em}

\bibitem {ARON} R. M. Aron, V. I. Gurariy and J. B. Seoane-Sep\'ulveda, {\it Lineability and
spaceability of sets of functions on $\R$}, Proc. Amer. Math. Soc. {\bf 133}
(2005), 795-–803.

\vspace{-0.5em}

\bibitem {Astala2}K. Astala, {\it On Peano's theorem in locally convex spaces},
Studia Math. \textbf{73} (1982), 213--223.

\vspace{-0.5em}

%\bibitem {Barroso1} C. S. Barroso, The approximate fixed point
%property in Hausdorff topological vector spaces, \textcolor{red}{??????????}.

\bibitem{LAA} G. Botelho, D. Diniz, V. F\'avaro and D. Pellegrino, {\it Spaceability in Banach and quasi-Banach sequence spaces}, Linear Algebra Appl. {\bf 434} (2011), 1255--1260.

\vspace{-0.5em}

\bibitem {Cellina}A. Cellina, {\it On the nonexistence of solutions of differential
equations in nonreflexive spaces}, Bull. Amer. Math. Soc. \textbf{78} (1972), 1069--1072.

\vspace{-0.5em}

\bibitem {Dieudonne}J. Dieudonn\'e, {\it Deux examples singuliers d'equations
differentielles}, Acta Sci. Math. (Szeged) \textbf{12B} (1950), 38--40.

\vspace{-0.5em}

\bibitem{pams} J. L. G\'amez-Merino, G. A. Mu\~noz-Fern\'andez, V. M. S\'anchez,  J. B. Seoane-Sep\'ulveda, {\it Sierpi\'nski-Zygmund functions and other problems on lineability}, Proc. Amer. Math. Soc. {\bf 138} (2010), no. 11, 3863--3876.

\vspace{-0.5em}

\bibitem{sinica} F. J. Garc\'{\i}a-Pacheco and J. B. Seoane-Sep\'ulveda, {\it Vector spaces of non-measurable functions}, Acta Math. Sin. (Engl. Ser.) {\bf 22} (2006), 1805--1808.

\vspace{-0.5em}

\bibitem {Godunov1}A. N. Godunov, {\it Counterexample to Peano's theorem in
infinite-dimensional Hilbert space}, Vestnik Mosko. Univ. Ser. 1 Mat.
Mekh. [Moscow Univ. Math. Bull.], \textbf{5} (1972), 19--21.

\vspace{-0.5em}

\bibitem {Godunov2}A. N. Godunov, {\it On Peano's theorem in Banach spaces},
Funktsional. Anal. i Prilozhen. [Functional Anal. Appl.], \textbf{9} (1975), 59--60.

\vspace{-0.5em}

\bibitem{Hajek-Johanis} P. H\'ajek and M. Johanis, {\it On Peano's theorem in Banach spaces}, J. Differential Equations {\bf 249} (2010), 3342--3351.
%\bibitem {Hajek-Johanis}P. H\'ajek and M. Johanis, On Peano's theorem in
%Banach spaces.

%\bibitem {Lin}P.-K. Lin and Y. Sternfeld, Convex sets with the Lipschitz fixed
%point property are compact, Proc. Amer. Math. Soc. {\bf 93} (1985), 633--639.

\vspace{-0.5em}

\bibitem{timoney} D. Kitson and R. M. Timoney, {\it Operator ranges and spaceability}, J. Math. Anal. Appl. {\bf 378} (2011), 680--686.

\vspace{-0.5em}

\bibitem{Lasota-Yorke} A. Lasota and J. A. Yorke, {\it The generic property of existence of solutions of differential equations in Banach spaces}, J. Differential Equations {\bf 13} (1973), 1--12.

\vspace{-0.5em}

\bibitem {Lobanov1} S. G. Lobanov, {\it Peano's theorem is invalid for any infinite-dimensional Fr\'echt space}, (Russian) Mat. Sb. {\bf 184} (1993), 83--86; translation in Russian Acad. Sci. Sb. Math. {\bf 78} (1994), 211--214.

\vspace{-0.5em}

\bibitem {Lobanov2} S. G. Lobanov, O. G. Smolyanov, {\it Ordinary differential equations in locally convex spaces}, (Russian) Uspekhi Mat. Nauk. {\bf 49} (1994), 93--168; translation in Russian Math. Surveys {\bf 49} (1994), 97--175.

\vspace{-0.5em}

\bibitem{PAMSUCM} J. L\'opez and S. Codes, {\it Vector spaces of entire funcions of unbounded type}, Proc. Amer. Math. Soc. {\bf 139} (2011), 1347--1360.

\vspace{-0.5em}

\bibitem {metafune}G. Metafune, V. B. Moscatelli, {\it On the space $l^{p+}%
=\bigcap\limits_{q>p}l^{q}$}, Math. Nachr. \textbf{147} (1990), 7--12.

\vspace{-0.5em}

\bibitem {Shkarin1} S. A. Shkarin, {\it Peano's theorem in infinite-dimensional Fr\'echet spaces is invalid}, (Russian) Funktsional. Anal. i Prilozhen {\bf 27} (1993), 90--92; translation in Funct. Anal. Appl. {\bf 27} (1993), 149--151.

\vspace{-0.5em}

\bibitem {Shkarin2} S. A. Shkarin, {\it Peano's theorem is invalid in infinite-dimensional $F'$-spaces}, (Russian) Mat. Zametki {\bf 62} (1997), 128--137: translation in Math. Notes {\bf 62} (1997), 108--115.

\vspace{-0.5em}

%\bibitem{Shkarin3} S. Shkarin, On Osgood theorem in Banach spaces, Math. Nachr. {\bf 257} (2003), 87--98.

%\bibitem {Teixeira}E.~V.~Teixeira\textcolor{red}{ N\~ao encontrei men\c c\~ao??????????},

\bibitem {Yorke}J. A. Yorke, {\it A continuous differential equation in Hilbert
space without existence}, Funkcial. Ekvac. \textbf{13} (1970), 19--21.


\end{thebibliography}
\end{document}